\documentclass[reqno, 12pt]{amsart}

\usepackage[linktocpage,colorlinks,linkcolor=magenta,citecolor=magenta,urlcolor=black,pagebackref=false]{hyperref}

\usepackage{amsmath,amssymb,amsthm}
\usepackage[T1]{fontenc}
\usepackage{lmodern}
\usepackage[noEnd=false]{algpseudocodex}
\usepackage{enumitem}
\usepackage[all]{xy}
\usepackage{microtype}
\usepackage{xcolor}
\usepackage{cleveref}  

\let\oldphi\phi
\let\phi\varphi
\let\varphi\oldphi
\let\oldepsilon\epsilon
\let\epsilon\varepsilon
\let\varepsilon\oldepsilon

\newcommand{\fm}{\mathfrak{m}}

\newcommand{\mbb}[1]{\ensuremath{\mathbb{#1}}}
\newcommand{\Z}{\mbb{Z}}

\newcommand{\PP}{\mbb{P}}
\newcommand{\Q}{\mbb{Q}}

\newcommand{\mbf}[1]{\ensuremath{\mathbf{#1}}}

\newcommand{\mcal}[1]{\ensuremath{\mathcal{#1}}}

\newcommand{\Fc}{\mcal{F}}

\newcommand{\tn}[1]{\textnormal{#1}}

\let\ker\relax\newcommand{\ker}{\tn{ker}}

\DeclareMathOperator{\depth}{\tn{depth}}

\DeclareMathOperator{\ini}{\tn{in}}
\DeclareMathOperator{\Proj}{Proj} 
\DeclareMathOperator{\reg}{reg}
\newtheorem{thm}{Theorem}[section]
\newtheorem{lemma}[thm]{Lemma}
\newtheorem{cor}[thm]{Corollary}
\newtheorem{prop}[thm]{Proposition}

\newtheorem{propdef}[thm]{Proposition and Definition}

\theoremstyle{definition}
\newtheorem{definition}[thm]{Definition}

\newtheorem{remark}[thm]{Remark}
\newtheorem{example}[thm]{Example}
\newtheorem{alg}[thm]{Algorithm}

\newtheorem*{thm*}{Theorem}

\numberwithin{equation}{section}
\algrenewcommand\algorithmicrequire{\textbf{Input:}}
\algrenewcommand\algorithmicensure{\textbf{Output:}}

\title{Macaulay Constants and Vanishing of Cohomology}

\author{Uwe Nagel}
\address{Department of Mathematics, University of Kentucky, 715 Patterson Office Tower, Lexington,
KY 40506,  USA}
\email{uwe.nagel@uky.edu}

\thanks{The author was partially supported by Simons Foundation grant \#636513. He thanks Kyle Maddox for helpful comments. }

\begin{document}

\begin{abstract}
Dub\'e introduced cone decompositions and their Macaulay constants and used them to obtain an upper bound on the degrees of the generators in a Gr\"obner basis of an ideal. Liang extended the theory to submodules of a free module. 
In this paper, Macaulay constants of any finitely generated graded module $M$ over a polynomial ring are introduced by adapting the concept of a cone decomposition to $M$. 
It is shown that these constants provide upper bounds for the degrees in which the local cohomology modules of $M$ are not zero. The results include an upper bound on the Castelnuovo-Mumford regularity of $M$ and a generalization of Gotzmann's Regularity Theorem from ideals to modules. As an application, an upper bound on the  Castelnuovo-Mumford regularity of any coherent sheaf on projective space is established. The mentioned bounds are sharp even for cyclic modules. Furthermore,  Macaulay constants are utilized to provide a characterization of Hilbert polynomials of finitely generated graded modules.  
\end{abstract}

\maketitle

\tableofcontents


\section{Introduction}
\label{section:intro}

In \cite{Dube}, Dub\'e introduced cone decompositions in order to bound the degree of polynomials appearing in reduced Gr\"obner bases of a polynomial ideal and expressed the hope that these decompositions have other applications as well. A goal of this paper is to provide further instances for the usefulness of cone decompositions. 

Consider any  finitely generated graded module $F$ over a  polynomial ring $S = K[X]$ over any field $K$ with finite variable set $X$ and standard grading, i.e., all the variables of $S$ have degree one. A \emph{cone} $C$ in $F$ is a graded $K$-subspace of the form $C = h K[U]$, where $h \in F$ is homogeneous and $U$ is a subset of $X$. For a graded subspace $P$ of $F$, any finite direct sum of cones
\[
P = \bigoplus_{i = 1}^t h_i K[U_i].  
\]
is called a \emph{cone decomposition} of $P$. Dube realized the importance of special cone decompositions, called $q$-exact,  satisfying additional properties  (see \Cref{def:q-exactness}). Extending a result of Dub\'e, Liang showed in \cite[Theorem 20]{Liang} that certain graded subspaces of $F$ admit a $q$-exact cone decomposition for a suitable integer $q$.  

Denote by $M$ any nonzero finitely generated graded $S$-module. 
We introduce the \emph{Macaulay constants} $b_0 (M),\ldots, b_{d+1} (M)$ of $M$, where $d$ is the  Krull dimension of $M$, by reading them off from suitable $q$-exact cone decompositions (see \Cref{sec:cone decomp modules}). It turns out that these integers satisfy 
\[
b_0 (M) \ge b_1 (M) \ge \cdots \ge b_d (M) > b_{d+1} (M) = e^+ (M),  
\] 
where $e^+ (M)$ denotes the maximum degree of a minimal generator of $M$ (see \Cref{prop:comp Macaulay const}). 

The Macaulay constants of $M$ determine the Hilbert polynomial of $M$ when it is written in a particular form. We use this form  to obtain an explicit characterization of Hilbert polynomials of finitely generated graded $S$-modules (see \Cref{thm:char hilb pol}). 

The Castelnuovo-Mumford regularity $\reg (M)$  is an important and much-investigated complexity measure of $M$. It can be defined using a graded minimal free resolution of $M$ over $S$ or the local cohomology modules $H^i_\fm (M)$ of $M$ with support in the homogeneous maximal ideal $\fm$ of $S$. Taking into account the latter prospective, one defines, 
following \cite{N-90}, the \emph{$k$-regularity}  $\reg_k (M)$ of $M$ as 
\[
\reg_k (M) = 
\min \{ m \in \Z \; \mid \; [H^i_\fm (M)]_j = 0 \text{ whenever } j > m-i \text{ and } i \ge k\}.
\]
Thus, upper bounds  on $\reg_k (M)$ correspond to vanishing results for local cohomology modules. 
Note that 
\[
\reg (M) = \reg_0 (M) \ge \reg_1 (M) \ge \cdots \ge \reg_{d} (M). 
\]

The following statement partially summarizes our cohomological vanishing results  (see \Cref{thm:reg-i bound} and \Cref{cor:sharpness} for more details). 

\begin{thm*} 
   \label{thm:intro reg-i bound} 
If $M \ne 0$ is any finitely generated graded $S$-module then one has 
 \[
 \reg_k (M) < b_k (M) \quad \text{  for } k=0,1,\ldots,\dim M.  
 \]
 
 Moreover, for any integers $0 \le t \le d \le \dim S$ and $b_t  \ge  b_{t+1}\ge \cdots \ge b_d > b_{d+1}$, there a is cyclic graded $S$-module $M$ satisfying $\dim M = d$, $b_i (M) = b_i$ for $k = t,t + 1,\ldots,d+1$ and 
 \[
 \reg_k (M) = b_k (M) - 1 \quad \text{whenever } t \le k \le d. 
 \]
\end{thm*}

It is worth stressing two special cases of this statement. For $k = 0$, one gets $\reg (M) \le b_0 (M)$. The estimate $\reg_1 (M) \le b_1 (M)$ may be interpreted as an extension of 
Gotzmann's Regularity Theorem in \cite{Go} from ideals to modules (see \Cref{rem:compare with Gotzmann}). 
Our results also imply an upper bound on the Castelnuovo-Mumford regularity of a coherent sheaf on $\PP^n$. It is attained for structure sheaves of certain projective subschemes (see \Cref{cor:reg sheaves}). 

This paper is organized as follows. In \Cref{sec:cone decomp modules}, we recall the concept of a cone decomposition  and use special cone decompositions to define the Macaulay constants of a finitely graded module $M$. We also establish some of their basic properties  (see \Cref{lemdef:Macaulay constants}). If the grading of $M$ is shifted its Macaulay constants change, but in a controlled way. 

In \Cref{sect:Hilbert pol}, we use Macaulay constants to present the Hilbert polynomial of a finitely graded module $M$ in a particular form. This leads to an explicit characterization of Hilbert polynomials of finitely graded modules. Subsequently, these results are used to establish further properties of Macaulay constants as, for example, their behavior under truncation and comparisons of Macaulay constants (see \Cref{prop:comp Macaulay const}). The latter result also relies on   \Cref{prop:reg index bound by Mac constants}. It guarantees that the Hilbert function and the Hilbert polynomial of $M$ are the same in any degree $j \ge b_0 (M)$. 

Our vanishing results on local cohomology modules are established in \Cref{sec:vanishing}. We begin by interpreting \cite[Theorem 7.5]{MN} on saturated ideals using a Macaulay constant and generalizing the statement to modules. The latter may be viewed as an extension of  Gotzmann Regularity Theorem. Using additional comparison results for Macaulay constants (see \Cref{prop:b_i and postive depth}), we establish the above theorem on $k$-regularities and explicitly describe modules for which the bounds are simultaneously sharp (see \Cref{thm:reg-i bound} and \Cref{cor:sharpness}).  We conclude with an application to coherent sheaves on projective space (see \Cref{cor:reg sheaves}).


\section{Review of Cone Decompositions} 
   \label{sec:cone decomp modules}

Dub\'e introduced cone decompositions for certain graded subspaces of a polynomial ring $S$ and Liang extended the concept to subspaces of a graded finitely generated free $S$-module. Our goal is to use cone decompositions for describing properties of finitely generated graded $S$-modules. We review cone decompositions from this point of view and introduce the Macaulay constants of a finitely generated graded module. 
   
We begin by fixing some notation. We denote by $K$ any field. A \emph{($\Z$-)graded $K$-vector space} is a $K$-vector space $V$ with a decomposition $V = \oplus_{j} [V]_j$ into subspaces. We refer to $[V]_j$ as the degree $j$ component of $V$. Assuming these components are all finite-dimensional, the \emph{Hilbert function} of $V$ is $h_V \colon \Z \to \Z, \  h_V (j) = \dim_K [V]_j$. 
The \emph{initial degree} of $V$ is  $a(V) = \inf \{j \in \Z \; \mid \; [V]_j \neq 0\}$, and so the initial degree of the zero module is $ \infty$. 

We denote by $S$ a polynomial ring over $K$ with finite variable set $X$ and standard grading, i.e., all variables of $S$ have degree one. We write $|U|$ for the cardinality of any finite set $U$. 

For a finitely generated  $\Z$-graded nonzero $S$-module  $M = \oplus_{j} [M]_j$, we denote by $e^+ (M)$ the maximum degree of a minimal generator of $M$. The \emph{truncation} $M_{\ge j}$ at degree $j$ is the submodule of $M$ consisting of all elements whose degree is at least $j$. The Hilbert function of $M$ is given by a polynomial in all sufficiently large degrees. This polynomial is called the \emph{Hilbert polynomial} of $M$ and denoted $p_M$. 

Denote by $F$ a finitely generated graded free $S$-module. 
Considering $F$ as  a graded $K$-vector space,  for any homogeneous $h \in F$ and any subset $U \subseteq X$, the set $C = h K[U]$ is a graded subspace of $F$. It is called a \emph{cone} generated by $h$. If $h \neq 0$  its initial degree is $a (C) = \deg h$, and $C$ is said to have dimension $|U|$. Note that the latter is equal to the Krull dimension when $C$ is considered as a $K[U]$-module. 
A \emph{cone decomposition} of a graded subspace $P$ of $F$ is any finite direct sum of cones 
\[
P = \bigoplus_{i = 1}^t h_i K[U_i].  
\]
We denote the maximum degree of a cone generator by $e^+ (P)$ , that is, 
\[
e^+ (P) = \max \{\deg h_i \; \mid \; i \in [t]\}. 
\]
The \emph{positive-dimensional part} $P^+$ of $P$ is the direct sum over the cones of $P$ with positive dimension. Observe that  $P^+ = 0$ if and only if $P$ is a finite-dimensional vector space. In this case, 
any cone decomposition of $P$ has the form $\bigoplus_{i=1}^{t} h_i K$, where $ t = \dim_K P$ and the number of distinct $h_i$ with degree $j$ is equal to $h_P (j)$. 
In general, the Hilbert function of a graded subspace $V$ imposes even fewer restrictions on cone decompositions of $V$. 

\begin{example}
   \label{exa:fan} 
For any cone $C =  h K[U]$, there is an equality of graded subspaces of $F$
\[
h K[U]  = h K \oplus \bigoplus_{k=1}^s x_{i_k} h K[x_{i_1},\ldots,x_{i_k}],     
\]
where $U$ consists of $s$ variables  $x_{i_1},\ldots, x_{i_s} \in X$. 
Following Dub\'e \cite{Dube}, the right-hand side is called the \emph{fan decomposition} of $C$ and denoted $F(C)$. Observe that in case $s \ge 1$ one has $a (F(C)) = a(C) = \deg h$ and $a (F(C)^+) = 1 + a(C^+) = 1 + \deg h$.   
\end{example}

Dub\'e introduced particular cone decompositions of subspaces of $S$. These were extended to subspaces of free $S$-modules by Liang. 

\begin{definition}
    \label{def:q-exactness} 
Consider a cone decomposition 
\[
P = \bigoplus_{i = 1}^t h_i K[U_i].  
\]

\begin{itemize}

\item[(i)] (\cite[Definition 12]{Liang}) The decomposition is said to be \emph{$q$-standard} if the following two conditions are satisfied: 
\begin{enumerate}

\item[(a)] For any $i \in [t]$, one has $h_i \neq 0$ and $U_i \neq \emptyset$ implies $\deg h_i \ge q$; 

\item[(b)]  For every $i \in [t]$ with $U_i \neq \emptyset$ and any integer $d$ such that $q \le d \le \deg h_i$, there is some $h_j$  ($j \in [t]$) 
with $\deg h_j = d$ and $|U_j| \ge |U_i|$. 
\end{enumerate}

\item[(ii)] (\cite[Definition 21]{Liang}) The decomposition is said to be \emph{$q$-exact} if it is $q$-standard and, for any $i \neq j$ with 
$U_i \neq \emptyset$ and $U_j \neq \emptyset$, 
one has $\deg h_i \neq \deg h_j$. 
\end{itemize}

\end{definition} 

\begin{remark}
    \label{rem:q-standard}
If $P$ is   a cone decomposition with $P^+ = 0$ then $P$ is $q$-exact for any integer $q$. However, if $P^+$ is non-trivial then one must have $q = a(P^+)$. 
\end{remark}

Dub\'e showed that any $q$-standard decomposition can be transformed to a $q$-exact cone decomposition by using the following strategy:  Whenever one has two cones with the same initial degree replace the cone of lower dimension by its fan decomposition. This leads to the following procedure. 

\begin{alg}[Converting a $q$-Standard to a $q$-Exact Cone Decomposition]
       \label{alg:transform q-stand to q-exact} \mbox{ }
       
Input: A $q$-standard cone decomposition $\tilde{P}$ of a graded vector space $V$ with  $(\tilde{P})^+ \neq \emptyset$. 

Output: A $q$-exact cone decomposition $P$ of $V$.

\begin{enumerate}

\item Set $P := \tilde{P}$. 

\item For $j = q, q+1,\ldots,e^+ (P^+)$, repeat the following steps: 

\begin{enumerate}

\item Denote by $T$ the set of positive-dimensional cones of $P$ with initial degree $j$. 

\item While $|T| \ge 2$ repeat the following steps:  Choose a cone $C$ in $T$ with minimum dimension and replace $C$ in $P$ by its fan decomposition $F (C)$.  
Denote the new decomposition of $V$ by $P$ and  update $T$ to $T \setminus \{C\}$.   

\end{enumerate}

\item Return $P$. 
\end{enumerate}

\end{alg}

We illustrate the algorithm by a simple example. 

\begin{example} 
   \label{exa:conversion q-stand to q-exact}
\Cref{alg:transform q-stand to q-exact} replaces the 1-standard decomposition $\tilde{P} = x K[x,y] \oplus y K[x,y]$ first by, say,  
\[
x K[x,y] \oplus y K \oplus xy K[x] \oplus y^2 K[x, y] 
\]
and then by 
\[
x K[x,y] \oplus y K \oplus xy K \oplus xy^2 K[x]  \oplus y^2 K[x, y].  
\]
It is 1-exact because the cones $yK$ and $x y K$ are $0$-dimensional. 
\end{example}

\begin{lemma}
    \label{lem:q-standard to q-exact}
\Cref{alg:transform q-stand to q-exact} terminates and is correct. 
\end{lemma} 

\begin{proof}
The assertions are consequences of the properties of fan decompositions. For details, see \cite[Lemma 16]{MR}. 
\end{proof}

As observed by Dub\'e and Liang, $q$-exact cone decompositions enjoy certain uniqueness properties. 

\begin{lemma}
     \label{lem:props q-exactness}
If $V$ is any infinite-dimensional graded subspace of $F$, then any $q$-exact cone decomposition of $V$ (if one exists) is of the form 
\begin{equation*}
P =   \bigoplus_{k=1}^{t} h_{k, 0} K \  \oplus \   \bigoplus_{j=1}^{d}  \ \bigoplus_{k=b_{j+1}}^{b_j -1} h_{k, j} K[U_j]
\end{equation*}
with homogeneous $h_{k, j} \in F$ and   integers $d$,  $b_1,\ldots,b_{d+1}$ and $t \ge 0$ that satisfy  
\begin{itemize} 

\item $h_{k, j} \ne 0$ and $\deg h_{k, j} = k$ if $j \ge 1$;  

\item $b_1 \ge \cdots \ge b_{d+1} = q$; and 

\item $|U_{ j}| = j$. 

\end{itemize} 

Furthermore, $d$,  $b_1,\ldots,b_{d+1}$ and $e^+(P)$ are uniquely determined by $q$ and the Hilbert function of $V$. 
\end{lemma}

\begin{proof}
The stated properties are a straightforward consequence of the definition of a $q$-exact cone decomposition (see \cite[Lemma 23]{Liang} and \cite[Lemma 6.1]{Dube}). 
For the convenience of the reader we indicate the main steps: 
By \Cref{rem:q-standard}, $q$ must be the initial degree of $P^+$. Thus, 
 $d$,  $b_1,\ldots,b_d$ are uniquely determined by the Hilbert function of $V$ (see \cite[Section 7]{Dube}).  
 Note that $b_1 = e^+ (P^+) + 1$. 
 Comparing vector space dimensions it follows now that $t$ and $e^+ (P)$  are uniquely determined as well. In particular, one has $e^+ (P) = b_1 - 1$ if $P = P^+$ (i.e.\  $t=0$) and 
 $e^+ (P)$ is equal to the maximum of $b_1 - 1$ and $\max \{ j \in \Z \; \mid \; h_V (j) = h_P (j) > h_{P^+} (j) \}$, otherwise.   
\end{proof}

Consider now any finitely generated graded $S$-module $M \neq 0$ and a homogeneous minimal generating set $G$ of $M$. By Nakayama's Lemma, the degrees of the elements $g_1,\ldots,g_r$ of $G$ are uniquely determined. Fix any graded  free $S$-module $F$ of rank $r = |G|$ with a basis of homogeneous elements $e_1,\ldots,e_r$ satisfying $\deg e_j = \deg g_j$. Thus, there is a graded epimorphism $\phi \colon F \to M$ that maps each $e_j$ onto $g_j$. Such a map $\phi$ is said to be a \emph{minimal epimorphism}. The module  $Q = \ker \, \phi$ is called a \emph{syzygy module} of $M$. Note that $F$ and $Q$ are uniquely determined up to graded isomorphisms by $M$. Thus, we denote them by $F_M$ and $Q_M$ if we want to stress the dependency on $M$. 

Fix now any monomial order $<$ on the monomials of $F$. For background on Gr\"obner bases, we refer to \cite[Chapter 15]{Eis95}.  Denote by $N_M$ the $K$-subspace of $F$ whose basis consists of the monomials in $F$ that are not in the initial module $\ini_< (Q)$ of $Q$. 
Observe that  $F = Q \oplus N_M$ as graded vector spaces  (see, e.g., \cite[Theorem 15.3]{Eis95}). 
Moreover, $M \cong N_M$ as graded vector spaces. For example, if $M$ is a graded cyclic module that is generated in degree zero then $M \cong S/I$ for some homogeneous ideal $I$,  and so $S/I \cong S/\ini_< (I) \cong N_M$ as graded vector spaces. 

Using the above notation, we extend the concept of cone decompositions of submodules of free modules to more general modules. 

\begin{definition}
    \label{def:cone decomp}
Let $M$ denote any finitely generated graded $S$-module.     
A \emph{cone decomposition of $M$} is any cone decomposition of $N_M \subseteq F$ where $F$ is a free $S$-module admitting  a minimal epimorphism  $F \to M$.  It is called $q$-standard or $q$-exact if the cone decomposition of $N_M$ has this property. 
\end{definition}

The definition of $N_M$ depends on the choices of $G$, $F$ and $<$. However, the isomorphism of graded $K$-vector spaces $N_M \cong M$ implies that the Hilbert function of $N_M$ is independent of these choices, which is crucial in view of \Cref{lem:props q-exactness}. 

Dub\'e established that, for any homogeneous ideal $I$ of $S$, the $K$-algebra $S/I$ admits a  $0$-exact cone decomposition. Reducing to this case, Liang essentially extended the result to graded $S$-submodules though she stated it differently (see {\cite[Theorem 20]{Liang}).  

\begin{prop} 
    \label{prop:existance q-decomp} 
Any finitely generated graded $S$-module $M$ admits an $e^+(M)$-exact cone decomposition. 
\end{prop}

\begin{proof}
For the reader's convenience, we sketch the argument which is essentially due to Liang. 

If $\dim M = 0$ there is nothing to show by \Cref{rem:q-standard}. 

Assume $d = \dim M \ge 1$. Set $e = e^+ (M)$.  Using the above notation, write $M \cong F/Q$. By construction of $F = F_M$, one has $e^+ (F) = e^+ (M) = e$. Fix a basis $\{e_1,\ldots,e_r\}$ of $F$ consisting of homogeneous elements. It follows that $\ini_< (Q) \cong \bigoplus_{k=1}^r I_k e_k$ as graded $S$-modules for suitable monomial ideals $I_k$ of $S$. According to \cite[Theorem 4.11]{Dube}, each $N_{S/I_k}$ admits a $0$-standard cone decomposition, and so $N_{S/I_k} e_k$ admits a $\deg (e_k)$-standard decomposition. Since $e$ is the maximum of $\deg (e_1),\ldots,\deg (e_r)$,  replacing successively any positive-dimensional cone of initial degree less than  $e$  in the decomposition of $N_{S/I_k} e_k$ by its fan decomposition, one obtains an $e$-standard decomposition of $N_{S/I_k} e_k$ (see \cite[Lemma 15]{Liang} for details). 
Putting these together, gives an  
$e$-standard decomposition of the subspace 
\[
N_M  =  \bigoplus_{k=1}^r N_{S/I_k} e_k. 
\]
Now one uses \Cref{alg:transform q-stand to q-exact} to transform this decomposition to an $e$-exact cone decomposition. 
\end{proof}

Dub\'e and Liang introduced Macaulay constants of a $q$-exact cone decomposition of a graded subspace of a free module.   We adapt the concept for a graded $S$-module $M$.  Recall that $e^+ (M)$ denotes the maximum degree of a minimal generator of $M$. 

\begin{propdef}
     \label{lemdef:Macaulay constants}
Let $M \neq 0$ be any finitely generated graded $S$-module. Denote its Krull dimension by $d = \dim M$.  Consider any $e^+ (M)$-exact cone decomposition of $M$, 
\begin{align}
    \label{eq:q-exact decomp, constants}
P =   \bigoplus_{k=1}^{t} h_{k, 0} K \  \oplus \   \bigoplus_{j=1}^{d}  \ \bigoplus_{k=b_{j+1}}^{b_j -1} h_{k, j} K[U_j].     
\end{align}
(Such a decomposition exits by \Cref{prop:existance q-decomp}.) 
Define 
\[
b_0  = 1 + e^+ (P). 
\]
If $d = 0$ then put $b_1  = e^+ (M)$. 

The integers $b_0,\ldots,b_{d+1}$ are uniquely determined by $e^+ (M)$ and the Hilbert function of $M$. Setting $b_i (M ) = b_i$, the numbers $b_0 (M),\ldots,b_{d+1} (M)$ are called the \emph{Macaulay constants of $M$}. 

Moreover, if $d \ge 1$ 
one has   $b_d (M) > b_{d+1}(M) = e^+ (M)$. 
\end{propdef} 

\begin{proof}
The claimed uniqueness is a consequence of \Cref{lem:props q-exactness} with $q = e^+ (M)$. 
 
If $d \ge 1$ then the Hilbert polynomial of $M$ has degree $d-1$. In $P$, the only summands with a Hilbert polynomial of degree $d-1$ involve $U_d$ as $|U_j| = j$ by \Cref{lem:props q-exactness}. It follows that $b_d - 1 \ge b_{d+1}$, as claimed. 
\end{proof} 

\begin{remark}
   \label{rem:Macaulay constants compare} 
(i) If $d = \dim M \ge 1$ observe that $b_0 (M) -1 = e^+ (P) \ge e^+ (P^+) = b_1 (M) -1$, and so one has   
\[
b_0 (M) \ge b_1 (M) \ge \cdots \ge b_d (M) > b_{d+1} (M).  
\]

(ii) If $d  \ge 1$, then notice that  the number of $j$-dimensional cones in any $e^+(M)$-exact cone decomposition of $M$ is equal to $b_j (M) - b_{j+1} (M)$ for every $j = 1,\ldots,d$. 
   
(iii) Regardless of $d$, one always has $b_{d+1} (M) = e^+ (M)$.

(iv) As above, write $M \cong F_M/P_M$ such that $F_M$ is free of  minimum rank. In \cite[Definition 22]{Liang}, Liang defines  Macaulay constants $b_i (P)$ of a $q$-exact cone decomposition $P$ of $P_M$. If $q = e^+ (M)$ and $d = \dim M \ge 1$ then the above $b_i (M)$ are equal to Liang's integers $b_i (P)$ for $i = 0,\ldots,d+1$. Note that $b_1 (M) = 1 + e^+ (P^+)$ and $b_{d+1} (M) = e^+ (M)$  in this case. 

If $d = 0$ one has again $b_0 (M) = b_0 (P)$. However, if $d=0$ Liang defines the decomposition $P$ to be 0-exact, which gives $b_1 (P) = 0$. Our definition of $b_1 (M) = e^+ (M)$ in this case allows for a more uniform treatment (see, e.g., \Cref{prop:comp Macaulay const}). 

(v) Dub\'e and Liang define Macaulay constants $b_0 (P),\ldots,b_{n+1} (P)$ of a $q$-exact cone decomposition  $P$. Since $b_{d+1} (P) = \cdots = b_{n+1} (P)$, where $d = \dim (P)$,  the constants $b_i (P)$ with $i > d+1$ do not carry additional information, and we restrict ourselves to defining $b_0 (M),\ldots,b_{\dim M +1} (M)$.
\end{remark}

\begin{example}
    \label{exa:Mac constants principal} 
Consider a principal ideal $I =\langle f \rangle $ of $S$ generated by a homogeneous polynomial $f \neq 0$ of degree $a$. 

(i) For the $S$-module $I$, one has $e^+ (I) = a$, and $a$-exact cone decompositions
\[
I = f S \cong x_0^a S, 
\]
and so $b_0 (I) = \cdots = b_{n+1} (I) = a+1 = 1 + b_{n+2} (I)$. 

(ii) An $0$-exact cone decomposition of $S$ as an  $S$-module is 
\[
S \cong x_0^a S \oplus \bigoplus_{k = 0}^{a-1} x_0^k K[x_1,\ldots.x_n]. 
\]
Hence, if $a \ge 1$ then $S/I$ has a $0$-exact cone decomposition 
$S/I = \bigoplus_{k = 0}^{a-1} x_0^k K[x_1,\ldots.x_n]$. Since $e^+ (S/I) = 0$ this gives 
$b_0 (S/I) = \cdots = b_n (S/I) = a$ and  $ b_{n+1} (S/I) = 0$. 
\end{example}

Shifting the grading of a module changes its Macaulay constants predictably. Recall that, for any integer $k$, the degree shifted module $M(k)$ has the same underlying module structure as $M$ but its grading is given by $[M(k)]_j = [M]_{k+j}$ for any $j \in \Z$. 

\begin{lemma}
    \label{lem:constants and shifting}
For any integer $k$ and any finitely generated graded module $M \neq 0$, one has $b_i (M(k)) = b_i (M) -k$ for $i=0,1,\ldots,1+\dim M$. 
\end{lemma}

\begin{proof}
Put $d = \dim M$ and consider any $e^+ (M)$-exact cone decomposition $M \cong   \bigoplus_{j=1}^{s} h_{j} K[U_j]$.  Shifting degrees, it gives
\[
M(k) \cong   \bigoplus_{j=1}^{s} h_{j} K[U_j] (k) = \bigoplus_{j=1}^{s} \tilde{h}_{j} K[U_j],    
\]
where $\tilde{h}_j$ is the same element as $h_j$, but its degree  is $\deg h - k$. Since $e^+ (M(k)) = e^+ (M) - k$ this shows that the last cone decomposition is $e^+ (M(k))$-exact, and so the claim follows. 
\end{proof}

Dub\'e had the remarkable insight that the degrees of the elements of a Gr\"obner basis of an ideal can be bounded above using cone decompositions (see \cite[Theorem 4.11]{Dube}). Liang established an extension to modules. 
We state it using our notation. 

\begin{thm}[{\cite[Theorem 20]{Liang}}]
    \label{thm:Groebner degree bound} 
If $Q \neq F$ is any graded submodule of a finitely generated graded $S$-module $F$ then  the degrees of the elements of a reduced Gr\"obner basis of $Q$ with respect to any monomial order on $F$ are at most $b_0 (F/Q)$. 
\end{thm}

If $Q = I$ is an ideal  $S$ Dub\'e's  Theorem 8.2 in \cite{Dube} may be viewed as an upper bound on $b_0 (S/I)$ by a function in $e^+ (I)$ and $n$. 
His estimate was improved in \cite{MR, HSM-21, HSM-22} and extended to a bound for $b_0 (F/Q)$ in \cite[Theorem 37]{Liang}.

Later,  we will see that Macaulay constants other than $b_0 (M)$ also provide information about a module $M$.


\section{Macaulay Constants and Hilbert Polynomials } 
    \label{sect:Hilbert pol}

In this section, we establish an explicit formula for the Hilbert polynomial of a finitely generated graded module $M$ in terms of its Macaulay constants. It is used to derive further properties of Macaulay constants and for guarantees that the Hilbert function and the Hilbert polynomial are the same in certain degrees. 

We begin with discussing consequences of cone decompositions for the Hilbert function and Hilbert polynomial of a graded module. We use the conventions that $\sum_{ j = a}^b c_j$ is defined to be zero if $b < a$ and, for integers $m, k$ with $k \ge 0$, 
\[
\binom{m}{k} = \begin{cases}
\frac{m (m-1) \cdots (m-k+1)}{k!} & \text{ if } m \ge k \ge 1;\\
1 & \text{ if } m \ge k = 0;\\
0 & \text{ if } m < k. 
\end{cases}
\]
In particular, as a polynomial in $\Q[z]$, one has $\deg \binom{z}{k} = k$ if $k \ge 0$. 

The Macaulay constants of a module can be used to describe its Hilbert polynomial in a specific form. 

\begin{prop}
   \label{prop:Hilb pol with Macaulay const}
The Hilbert polynomial $p_M \in \Q[z]$ of a finitely generated graded $d$-dimensional $S$-module $M$ with Macaulay constants $b_1,\ldots,b_{d+1}$ is 
\begin{align}
    \label{eq:hilbPol with MacConstants-1} 
p_M (z)  = \sum_{j=1}^d \sum_{k = b_{j+1}}^{b_j - 1} \binom{z-k+j-1}{j-1}, 
\end{align}
which can also be written as 
\begin{align}
    \label{eq:hilbPol with MacConstants-2} 
p_M (z)  = \sum_{j=1}^d \bigg[\binom{z - e^+ (M)+j-1}{j} -  \binom{z-b_j+j-1}{j} \bigg]. 
\end{align}
Moreover, if $ d\ge 1$ then the Macaulay constants $b_1 (M),\ldots,b_d(M)$ are uniquely determined by $e^+ (M)$ and $p_M$. 
\end{prop}

\begin{proof}
It is well-known that $d = 0$ if only if $p_M$ is the zero polynomial and that $p_M$ has degree $d-1$ if $d \ge 1$. 
Thus, if $d=0$ both claims are true since the sums  are defined to be zero if $d < 1$. 

Assume $d  \ge 1$. Since a cone $C = h K[U]$ with $U \neq \emptyset$ has Hilbert polynomial $p_C (z) = \binom{z- \deg h + |U|-1}{|U| -1}$, any $e^+ (M)$-exact cone decomposition \eqref{eq:q-exact decomp, constants} of $M$ immediately gives 
\begin{align}
p_M (z) &  =  \sum_{j=1}^d \ \sum_{k = b_{j+1}}^{b_j - 1} \binom{z- \deg (h_{k, j}) +|U_j|-1}{|U_j|-1} \\
& = \sum_{j=1}^d \sum_{k = b_{j+1}}^{b_j - 1} \binom{z-k+j-1}{j-1},  
\end{align}
where we use $\deg (h_{j, k}) = k$ and $|U_j| = j$  for the second equality (see \Cref{lem:props q-exactness}).    

Next, we establish Equation \eqref{eq:hilbPol with MacConstants-2} (see also \cite[Section 7]{Dube} for a similar computation). 
For any integers $c \ge 0$ and $s \gg 0$, one has the following well-known identity 
\[
\sum_{k=0}^{c-1} \binom{s+k}{k} = \binom{s+c}{c-1}, 
\]
and so for polynomials in $\Q[z]$, 
\begin{align}
    \label{eq:easy identity}
\sum_{k=0}^{c-1}  \binom{z+k}{k} = \binom{z+c}{c-1}. 
\end{align} 
Considering now integers $c \ge a$, $i \ge 0$ and $s \gg 0$, one computes
\begin{align*}
\sum_{k=a}^{c-1}  \binom{s-k+i}{i} & = \sum_{j = s-c+1}^{s-a} \binom{j+i}{i} \\
& = \binom{i + s-a+1}{s-a} - \binom{i + s-c+1}{s-c} \\
& =  \binom{ s-a+i+1}{i+1} - \binom{s-c+i+1}{i+1}, 
\end{align*}
where one uses the first identity above to get the second line. For polynomials in $\Q[z]$, the latter identity implies
\[
\sum_{k=a}^{c-1}  \binom{z-k+i}{i} =  \binom{ z-a+i+1}{i+1} - \binom{z-c+i+1}{i+1}. 
\]
Applying the last identity to the inner sums in Equation \eqref{eq:hilbPol with MacConstants-1}, one gets
\begin{align*}
p_M (z) 
& = \sum_{j=1}^d  \bigg[\binom{z - b_{j+1}+j}{j} -  \binom{z-b_j+j}{j} \bigg]. 
\end{align*}
Combining consecutive summands beginning with the second summand gives
\begin{align*}
p_M (z) & = \binom{z - b_{d+1}+d}{d} - 1 -  \sum_{j=1}^d  \binom{z - b_{j}+j-1}{j} \\
& = \sum_{j=1}^{d} \binom{z - b_{d+1} -1 + j}{j } -  \sum_{j=1}^d  \binom{z - b_{j}+j-1}{j},    
\end{align*}
where one invokes Identity \eqref{eq:easy identity} to obtain the second line. 
Recalling that $b_{d+1} (M) = e^+ (M)$, the desired Equation \eqref{eq:hilbPol with MacConstants-2} follows. 

The final assertion follows by observing that each summand $\big[\binom{z - e^+ (M)+j-1}{j} -  \binom{z-b_j+j-1}{j} \big]$ in Equation \eqref{eq:hilbPol with MacConstants-2} 
is either zero or has degree $j-1$. 
\end{proof} 

As an immediate consequence, one identifies the weighted top coefficient of the Hilbert polynomial of $M$, which is called the \emph{degree} or \emph{multiplicity} of $M$ and denoted $\deg M$. 

\begin{cor}
    \label{cor:identify b_d}
If $d = \dim M \ge 1$ then $b_d (M) = \deg M + e^+ (M)$.
\end{cor}

\begin{proof}
It is well-known that $p_M (z)$ can be written as $\deg M \cdot \binom{z}{d-1}$ plus lower order terms in $z$. Hence Equation~\eqref{eq:hilbPol with MacConstants-1}  gives $\deg M = b_d (M) - b_{d+1} (M) = b_d (M) - e^+ (M)$, where we also used  \Cref{lemdef:Macaulay constants}. 
\end{proof}

If $M = S/I$ is a graded $K$-algebra then one often abuses notation and defines $\deg I = \deg S/I$. Thus, \Cref{cor:identify b_d} yields
$b_d (S/I) = \deg I$ if $\dim S/I \ge 1$.

Truncating a module can considerably change the Macaulay constants. 

\begin{example}
Consider the module $M = x S \cong S(-1)$ with  $S = K[x, y]$. 
Its Macaulay constants are 
$b_0 (M) = b_1 (M) = b_{2} (M) = 2$ and $b_{3} (M) = e^+ (M) = 1$. 
The fan decomposition of $x S$ shows that the truncation $M_{\ge 2} = x \cdot (x, y)$  has a 2-standard cone decomposition $x^2 K[x] \oplus xy K[x, y]$. 
\Cref{alg:transform q-stand to q-exact} converts it to a 2-exact cone decomposition 
\[
M_{\ge 2} \cong x^2 K  \oplus x^3 K[x] \oplus xy K[x,y]. 
\]
Hence, the Macaulay constants of $M_{\ge 2}$ are 
$b_{3} (M_{\ge 2})  = e^+ (M_{\ge 2} ) = 2$, \ $b_2 (M_{\ge 2}) = 3$  and $b_1 (M_{\ge 2})  = b_0 (M_{\ge 2}) = 4$. 
\end{example}

However, truncating in low enough degrees, the Macaulay constants with positive index remain the same. This is a special case of the 
following observation whose argument relies on \Cref{prop:Hilb pol with Macaulay const}.  The result  will be used later on. 

\begin{lemma}
   \label{lem:Macaulay cont for direct sum}
Consider finitely generated graded $S$-modules $M$ and $N$. Then one has:
\begin{itemize}

\item[{\rm (a)}]  If $N$ is a zero-dimensional submodule of $M$ and $e^+ (M/N) = e^+ (M)$ then $b_i (M/N) = b_i (M)$ for every $i = 1,2,\ldots, \dim M + 1$. 

\item[{\rm (b)}] If $\dim N = 0$ and 
$e^+ (N) \le  e^+ (M)$  then $b_i (M \oplus N) = b_i (M)$ for every $i = 1,2,\ldots, \dim M + 1$. 

\item[{\rm (c)}] If $\dim N \ge 1$ and 
$e^+ (N) \le  e^+ (M)$ then 
$
b_i (M) \le b_i (M \oplus N)  \text { for every }  i = 1,\ldots, \dim M + 1. 
$ 
\end{itemize} 

\end{lemma} 

\begin{proof}

Claims (a) and (b) follow by \Cref{prop:Hilb pol with Macaulay const} 
because the assumptions guarantee that  $M/N$ and $M \oplus N$ have the same Hilbert polynomial as $M$ and the maximum degree of a minimal generator of $M/N$ as well as  of 
$M \oplus N$ is $e^+ (M)$. 

It remains to show Assertion (c). The assumption on $e^+ (N)$ implies that $e^+ (M \oplus N) = e^+ (M)$, which we denote by $e$. 
By (a), passing from $M$ to the truncation $M_{\ge e}$ does not change the Macaulay constants with positive index. This is analogously true when passing from $M \oplus N$ to $(M \oplus N)_{\ge e}$. 
 Thus, we may assume that $e = e^+ (M) = e^+ (N)$. 
 
Denote by $P$ and $Q$ an $e$-exact cone decomposition of $M$ and $N$, respectively. Thus, $P \oplus Q$ is an $e$-standard cone decomposition of $M \oplus N$, which we convert to an $e$-exact cone decomposition $\mathcal{P}$ using \Cref{alg:transform q-stand to q-exact}. If this algorithm replaces a cone $C$ by its fan decomposition $F(C)$, then the initial degree $a(C)$ of $C$ is the minimum degree $k$ such that there are two cones in the composition at hand with the same dimension and initial degree $k$, and $C$ is one of them. In $F(C)$, there is exactly one cone of dimension $j$ for every $j = 1,\ldots,\dim C$, and all these positive-dimensional cones have initial degree $1 + k = 1+ a(C)$. Taking into account \Cref{rem:Macaulay constants compare}(ii), this implies Claim (c). 
\end{proof}

The above results lead to a characterization of Hilbert polynomials of finitely generated graded $S$-modules. Recall that the Hilbert polynomial of such a module $M$ is zero if and only if $\dim M = 0$ or, equivalently, $M$ is finite-dimensional $K$-vector space.

\begin{thm}
   \label{thm:char hilb pol}
For any integer $e$ and any polynomial $p \neq 0$ in $\Q[z]$, the following conditions are equivalent: 
\begin{itemize}

\item[{\rm (a)}] $p (z) =  \sum_{j=1}^d \big[\binom{z - e+j-1}{j} -  \binom{z-b_j+j-1}{j} \big]$ 
with integers $d \ge 1$ and  $b_1 \ge \cdots \ge b_d >  e$; 

\item[{\rm (b)}] There is a graded module $M$ with $e^+ (M) = e$, $d = \dim M \ge 1$ and an $e$-exact cone decomposition 
$M \cong \bigoplus_{k=1}^{t} h_{k, 0} K \  \oplus \   \bigoplus_{j=1}^{d}  \ \bigoplus_{k=b_{j+1}}^{b_j -1} h_{k, j} K[U_j]$, where $b_{d+1} = e$. 

\item[{\rm (c)}] There is a graded module $M$  with with $e^+ (M) = e$, $d = \dim M \ge 1$ and Macaulay constants $b_i (M) = b_i$ for $i=1,\ldots,d$  \; 

\item[{\rm (d)}] There is a graded module $M$  with $e^+ (M) = e$, $d = \dim M \ge 1$ and Hilbert polynomial $p_M = p$; and 

\item[{\rm (e)}] There is a graded cyclic module $M$  with $e^+ (M) = e$, $d = \dim M \ge 1$ and Hilbert polynomial $p_M = p$.  
\end{itemize}

\end{thm}

\begin{proof}
Note that (b) implies (c) by \Cref{lemdef:Macaulay constants}. 

\Cref{prop:Hilb pol with Macaulay const} shows that (c) gives (d).  

Condition (a) is a consequence of (d)  because the Macaulay constants satisfy $b_1 (M) \ge \cdots \ge b_{d} (M) > e^+ (M)$  by \Cref{lem:props q-exactness} and Proposition~\ref{lemdef:Macaulay constants}. 

Next, we show that (a) implies (e). Passing from $p(z)$ to the polynomial $p (z+e)$ and using \Cref{lem:constants and shifting}, it suffices to find a homogeneous ideal $I$ of a polynomial ring $S$ such that $p_{S/I}$  is equal to  $q(z) = \sum_{j=1}^d \big[\binom{z+j-1}{j} -  \binom{z-\tilde{b}_j+j-1}{j} \big]$ 
with integers $\tilde{b}_1 \ge \cdots \ge \tilde{b}_d >  0$. For example, the lexicographic ideal $L_q$ to $q$ has this property (see \cite[Theorem 2.3]{MN}). 

Obviously, (d) is a consequence of (e). 

(d) implies (b) by the existence result of \Cref{prop:existance q-decomp} and  the uniqueness properties established in Proposition~\ref{lemdef:Macaulay constants}. 
\end{proof}

The equivalence of (a) and (d) in \Cref{thm:char hilb pol} could, with some work, also be derived from the characterization of Hilbert functions of graded modules (see \cite[Theorem 4.2]{Gasharov}) that generalizes Macaulay's characterization for graded $K$-algebras. However, we are not aware that the explicit form of  Hilbert polynomials as given in (a)   has been previously established in the literature in the above generality.

Following \cite[Definition 18]{torino-school}, we recall the following concept. 

\begin{definition}
   \label{def:reg index}
The \emph{regularity index} of a finitely generated graded $S$-module $M \ne 0$ is the number 
\[
r (M) = \min \{k \in \Z \; \mid \; h_M (j) = p_M(j) \text{ for any } j \ge k \}. 
\]
For the zero module $M$, we set $r (M) = - \infty$. 
\end{definition}

For a homogeneous ideal $I$ of $S$, Mayr and Ritscher \cite{MR} refer to the regularity index of $S/I$ as the Hilbert regularity of $I$ and denote it by $\reg (I)$. However, we reserve  $\reg$ for denoting the Castelnuovo-Mumford regularity. As observed below, the two concepts are related (see \Cref{prop:compare reg and reg index}).  We first note though that the regularity index is bounded above by a Macaulay constant. 

\begin{prop}
    \label{prop:reg index bound by Mac constants}
For any finitely generated graded $S$-module $M \neq 0$, one has 
\begin{itemize}

\item[{\rm (a)}] If $\dim M = 0$ then $r(M) = b_0 (M)$. 

\item[{\rm (b)}] If $\dim M \ge 1$ then $b_0 (M) = \max \{ r(M), b_1 (M) \}$. 
\end{itemize}
In particular, one always has $r (M ) \le b_0 (M)$. 
\end{prop} 

\begin{remark}
In the special case where $M = S/I$, the estimate $b_0 (M) \le \max \{ r(M), b_1 (M) \}$ has been shown in \cite[Lemma 23]{MR}. 
\end{remark}

\begin{proof}[Proof of \Cref{prop:reg index bound by Mac constants}]

We begin by establishing $r(M) \le b_0 (M)$. To this end
consider any cone $C = h K[U]$ appearing in an $e^+ (M)$-exact cone decomposition $P$ of $M$ (see \Cref{lemdef:Macaulay constants}). If $U \ne \emptyset$ its Hilbert polynomial is $p_C (z) = \binom{z- \deg h + |U|-1}{|U| -1}$. Hence $p_C (j)$ is equal to $h_C (j)$ if $j \ge \deg h - |U| + 1$.  Since, by definition, $b_i (M) - 1$ is the maximum degree of a cone with dimension $i = |U|$ appearing in the decomposition $P$, we conclude that $p_C(j) = h_C (j)$ if $j \ge b_i (M) - i$ and $|U| = i \ge 1$. In particular, this shows that if $P^+ \neq 0$ then one gets for its Hilbert function
\begin{equation}
    \label{eq:rex index postive part}
h_{P^+} (j ) = p_{P^+} (j) = p_M (j) \quad \text{ if } j \ge b_1 (M) -1. 
\end{equation}

If $U = \emptyset$ then the Hilbert function of $C$ is
\begin{equation}
    \label{eq:hilb finite cone}
h_C (j) = \begin{cases}
1 & \text{ if } j = \deg h, \\
0 & \text{ otherwise}. 
\end{cases}
\end{equation}
Since $p_C$ is the zero polynomial this gives $h_C (j) = p_C (j)$ if $j \ge 1 + \deg h$. By definition of $b_0 (M)$, we have $\deg h < b_0 (M)$. Combined with Equation~\eqref{eq:rex index postive part} and $b_0 (M) \ge b_1 (M)$,   
it follows that $h_M (j) = p_M (j)$ whenever $j \ge b_0 (M)$, and so 
$r(M) \le b_0 (M)$. 

Now, we prove Claim (a). If $\dim M = 0$ then $P^+ = 0$. Since, by definition, $b_0 (M) = 1 + e^+ (P)$, there is a cone $C = h K$ with $\deg h = b_0 (M) -1$ in $P$. Thus, Formula \eqref{eq:hilb finite cone} shows that $h_M (b_0 (M) -1) \ge 1 >  0 = p_M (b_0 (M) -1)$, which implies $r (M) \ge b_0 (M)$, and so the claimed equality follows. 

For establishing Claim (b), 
we consider two cases. 

If $r (M) > b_1 (M)$ then Equation~\eqref{eq:rex index postive part} shows that there must be a cone $C = h K$ in $P$ with $\deg h = r (M) -1$ and there is no zero-dimensional cone in $P$ whose initial degree  is at least $r(M)$. Since $e^+ (P^+) \le b_1 (M) - 1$, the assumption $r (M) > b_1 (M)$ implies  that $e^+ (P) = r(M)-1$, and so 
$b_0 (M) = e^+ (P) + 1 = r(M)$, as desired. 

If $r (M) \le b_1 (M)$ then, using again Equation~\eqref{eq:rex index postive part}, we obtain for $j \ge b_1 (M)$, 
\[
p_{P^+} (j) = h_{P^+} (j) \le h_M (j) = p_M (j) = p_{P^+} (j), 
\]
and so $h_{P^+} (j) = h_M (j)$, which shows that there is no zero-dimensional cone in $P$ whose initial degree is at least $b_1 (M)$. We conclude that $e^+ (P) = b_1 (M) -1$, and so $b_0 (M) = e^+ (P) +1 = b_1 (M)$, as claimed. 
\end{proof}

Denote the $i$-th local cohomology module of an $S$-module $M$ with support in the homogeneous maximal ideal $\fm$ of $S$ by 
$H^i_\fm (M)$. It is a graded module if $M$ is. 
The \emph{Castelnuovo-Mumford regularity} of $M \neq 0$ is 
\[
\reg (M) = \min \{ m \in \Z \; \mid \; [H^i_\fm (M)]_j = 0 \text{ whenever } j > m-i \}.
\]
This can be rewritten more concisely by defining the \emph{end} of a graded vector space $V$ as 
\[
e(V) = \sup \{j \in \Z \; \mid \; [V]_j \neq 0 \}. 
\]
In particular, $e(0) = - \infty$. Thus, one has 
\[
\reg (M) = \max \{ i + e (H^i_\fm (M) \; \mid \; i \ge 0 \}. 
\]
Following \cite{N-90}, we define the \emph{$k$-regularity}  $\reg_k (M)$ of $M$ as 
\[
\reg_k (M) = \max \{ i + e (H^i_\fm (M) \; \mid \; i \ge k \}. 
\]
Note though that in \cite{N-90} these integers were denoted $r_k (M)$.  From the definitions, one has 
\[
\reg (M) = \reg_0 (M) \ge \reg_1 (M) \ge \cdots \ge \reg_{dim M} (M). 
\]
Moreover, if $t = \depth M$ and so $H^i_\fm (M) = 0$ for $i < t$ 
 then $\reg (M) = \reg_t (M)$. 
 
The Castelnuovo-Mumford regularity can also be defined using syzygies.  Let 
\[
\cdots \to F_1 \to F_0 \to M \to 0 
\]
be a graded minimal free resolution of $M$ over $S$ with $F_i = \oplus_j S(-a_{i, j})$. Then $\reg (M)$ is the maximum of the integers $a_{i, j} - i$ (see \cite{EG}). In particular, one has $e^+ (M) \le \reg (M)$. Furthermore, 
for any homogeneous ideal $I$ of $S$, this implies immediately $\reg (I) = 1 + \reg (S/I)$.

The following observation describes the mentioned comparison between the regularity index and the Castelnuovo-Mumford regularity.  It is stated explicitly 
in \cite{torino-school} (see also \cite[Theorem 4.2 and Corollary 4.8]{Eis05} for related results).

\begin{prop}[{\cite[Lemma 8]{torino-school}}] 
    \label{prop:compare reg and reg index}
If $M \ne 0$ is any finitely generated graded $S$-module then one has 
\[
\reg (M) - \dim M + 1 \le r(M) \le  \reg (M) - \depth M + 1. 
\]
\end{prop} 

Now an extension of \Cref {rem:Macaulay constants compare}(i) follows. 

\begin{prop}
   \label{prop:comp Macaulay const} 
If $M$ is any $d$-dimensional, finitely generated graded $S$-module then one has 
\[
b_0 (M) \ge b_1 (M) \ge \cdots \ge b_d (M) > b_{d+1} (M) = e^+ (M). 
\] 
\end{prop}

\begin{proof}
If $d \ge 1$ this is true by Proposition~\ref{lemdef:Macaulay constants} and \Cref {rem:Macaulay constants compare}(i). 

Assume $d=0$. Then, by definition, $b_1 (M) = e^+ (M)$. Combining \Cref{prop:reg index bound by Mac constants}(a) and \Cref{prop:compare reg and reg index}, we have $\reg (M) + 1 = r(M) = b_0 (M)$. Since one always has $e^+ (M) \le \reg (M)$, we conclude $b_0 (M) - 1 \ge b_1 (M)$, which completes the argument. 
\end{proof}


\section{Regularity Bounds} 
    \label{sec:vanishing}
    
In this section, we use Macaulay constants to establish vanishing results for the local cohomology modules of a finitely generated graded module.

We begin by recalling a statement from \cite{MN}.  

\begin{thm}[{\cite[Theorem 7.5]{MN}}]
    \label{thm:MN} 
Consider a saturated ideal    $I \ne S$ of $S$ and write the Hilbert polynomial of $S/I$ uniquely as 
\[
p_{S/I} (z)  = \sum_{j=0}^{d-1} \bigg[\binom{z +j-1}{j} -  \binom{z-c_j+j-1}{j} \bigg]
\] 
with integers $c_0 \ge c_1 \ge \cdots \ge c_{d-1} > 0$ where $d = \dim (S/I)$.  Then one has 
\[
\reg I \le c_0. 
\]

\end{thm} 

This statement is shown by a combinatorial argument. The key is an analysis of the number of expansions needed to produce any saturated strongly stable monomial ideal of $S$ with a given Hilbert polynomial by starting from a suitable identity ideal (see \cite[Theorem 4.4]{MN}). 

Combined with \Cref{prop:Hilb pol with Macaulay const}, the last result can be restated using a Macaulay constant. 

\begin{thm}
    \label{thm:reg1 bound for algebras} 
For any standard graded $K$-algebra $A$ of positive dimension, one has $\reg_1 (A) < b_1 (A)$. 
\end{thm} 

\begin{proof}
Write $A = S/I$ with a homogeneous ideal $I$ of a polynomial ring $S$. Denote by $I^{sat}$ the saturation of $I$. Then $A$ and $S/I^{sat}$ have the same Hilbert polynomial, and $A$ is generated in degree zero.  Thus, using the notation of \Cref{thm:MN}, a comparison with \Cref{prop:Hilb pol with Macaulay const}, where $e^+ (A) = 0$, gives $b_{i+1} (A) = c_i$ for $i = 0,1,\ldots,\dim M - 1$ and so $\reg I^{sat} \le b_1 (A)$. Observe that we used here that $b_1 (A) \ge \cdots \ge b_{\dim A +1} (A) > e^+ (A) = 0$, which is true by Proposition~\ref{lemdef:Macaulay constants}. 
Since $\reg_1 (A) = \reg (S/I^{sat}) = \reg (I^{sat}) - 1$, the claim follows. 
\end{proof}

Before generalizing this result to modules let us observe that \Cref{thm:MN} (and so \Cref{thm:reg1 bound for algebras}) is in fact equivalent to a well-known result. This realization was one of the starting points of this paper. 
Recall the following  result of Gotzmann \cite{Go} (see, for example, \cite[Theorem 4.3.2]{BH-book} for a proof in English). 

\begin{thm}[Gotzmann's Regularity Theorem] 
    \label{thm:Gotzmann reg}
Write the Hilbert polynomial of $A = S/I$ in the unique form 
\[
p_A (z) = \binom{z+a_1}{a_1}  + \binom{z+a_2 -1}{a_2} + \cdots + \binom{z+a_s -(s-1)}{a_s}
\]
with integers $a_1 \ge a_2 \ge \cdots \ge a_s \ge 0$. Then the saturation $I^{sat}$ of $I$ satisfies $\reg I^{sat} \le s$. 
\end{thm}

We need the following observation. 

\begin{lemma}
   \label{lem:compare hilb}
If a nonzero polynomial $p \in \Q[z]$ can be written as 
\begin{align*}
p (z) & = \binom{z+a_1}{a_1}  + \binom{z+a_2 -1}{a_2} + \cdots + \binom{z+a_s -(s-1)}{a_s} \\
& = \sum_{j=0}^{d-1} \bigg[\binom{z +j-1}{j} -  \binom{z-c_j+j-1}{j} \bigg] 
\end{align*}
with integers $a_1 \ge a_2 \ge \cdots \ge a_s \ge 0$ and $c_0 \ge c_1 \ge \cdots \ge c_{d-1} > 0$ for some $d \ge 1$,  then $s = c_0$. 
\end{lemma} 

\begin{proof}
It is well-known that, for any integers $d \ge 1$ and $c_0 \ge c_1 \ge \cdots \ge c_{d-1} > 0$, there is a homogeneous ideal $I \subset S$ such that $A = S/I$ has Hilbert polynomial 
\[
p_A (z) = p(z) =  \sum_{j=0}^{d-1} \bigg[\binom{z +j-1}{j} -  \binom{z-c_j+j-1}{j} \bigg]. 
\]
For example, one may choose for $I$ the lexicographic ideal $L_p$ to $p$ (see \cite[Theorem 2.3]{MN}). 
Comparing with \Cref{prop:Hilb pol with Macaulay const}, it follows for the Macaulay constants $b_i = b_i (A)$ of $A$ that $b_{i+1} = c_i$ for $i = 0,1,\ldots,d-1$. Since $e^+ (A) = 0$, 
writing $p_A$ in the form of Equation \eqref{eq:hilbPol with MacConstants-1}, we get 
\[
\binom{z+a_1}{a_1}  + \binom{z+a_2 -1}{a_2} + \cdots + \binom{z+a_s -(s-1)}{a_s}  = \sum_{j=1}^d \bigg[\binom{z +j-1}{j} -  \binom{z-b_j+j-1}{j} \bigg].
\]
By \Cref{rem:Macaulay constants compare}(ii), any 0-exact cone decomposition of $A$ has exactly $b_j - b_{j+1}$ cones of dimension $j$ for  $j = 1,2,\ldots,d$, where $b_{d+1} = e^+ (A) = 0$. This forces 
$a_1 = a_2 = \cdots = a_{b_d} = d-1$, \
$a_{b_d + 1} = \cdots = a_{b_{d-1} } = d-2$, \ ..., $a_{b_2 + 1} = \cdots = a_{b_{1} } = 0$ and $s = b_1 = c_0$. 
\end{proof} 

\Cref{lem:compare hilb} immediately shows that Gotzmann's Regularity Theorem is equivalent to \Cref{thm:MN}. In the form of \Cref{thm:reg1 bound for algebras}, these results can be generalized as follows. 

\begin{thm}
    \label{thm:reg-1 modules}
Any finitely generated graded $S$-module $M$ of positive dimension satisfies $\reg_1 (M) < b_1 (M)$. 
\end{thm}

\begin{proof} Passing from $M$ to $M_{\ge e^+ (M)}$ does not change $\reg_1$ nor $b_1$ (see \Cref{lem:Macaulay cont for direct sum}(a)). Thus, we may assume that $a = a(M) = e^+ (M)$. 

Write $M = F/Q$, where $F = S^r (-a) = \oplus_{j=1}^r S e_j$, where $r = h_M (a)$ and the basis elements $e_j$ of $F$ have degree $a$. 
Semicontinuity gives $\reg_1 (Q) \le \reg_1 (\ini_< (Q))$ for any initial module of $Q$ with respect to a monomial order $<$ on $F$. Note that $\ini_< (Q) \cong \oplus_{j=1}^r I_j e_j$ for some  monomial ideals $I_j$ of $S$. 
Since taking a direct sum commutes with local cohomology, it follows that 
\begin{align*}
\reg_1 (M) &   \le \reg_1 (F/\ini_< (Q)) = \reg_1 (\oplus_{j=1}^r (S/I_j) e_j) \\
& \le \max \{ \reg_1 ((S/I_j) e_j) \; \mid \; j \in [r] \} \\
& = \max \{ \reg_1 ((S/I_j) e_j) \; \mid \; j \in [r] \} \text{ and } \dim S/I_j \ge 1\}  \\
& < \max \{ b_1 ((S/I_j )e_j) \; \mid \; j \in [r] \text{ and } \dim S/I_j \ge 1 \}, 
\end{align*}
where the third line is true because $\reg_1 (N)$ of any zero-dimensional module $N$ is equal to $-\infty$ by definition. 
Using first Part (c) of \Cref{lem:Macaulay cont for direct sum} and then its Part (b), we get 
\begin{align*}
\max \{ b_1 ((S/I_j) e_j) \; \mid \; j \in [r] \text{ and } \dim S/I_j \ge 1 \} & \le b_1  (\bigoplus_{j \in [r] \text{ and } \dim S/I_j \ge 1 }   (S/I_j) e_j ) 
\\[2pt]
& =  b_1  (\bigoplus_{j \in [r]}   (S/I_j) e_j )\\[2pt]
& = b_1 (F/\ini_< (Q)) \\
& = b_1 (F/Q) = b_1 (M). 
\end{align*}
Combining the previous two estimates, the claim follows. 
\end{proof} 

Observe that the largest submodule of $M$ with finite length is precisely $H^0_\fm (M)$. Following \cite{MNP}, we call $M^{sm} = M/H^0 (M)$ the  \emph{slight modification} of $M$. If $M = S/I$ then $M^{sm} \cong S/J$, where $J \subset S$ is the saturation of the ideal $I$. More generally, if $M = F/Q$ for some free $S$-module $F$ then 
\[
Q^{sat} = \bigcup_k (Q :_F \fm^k)
\]
is the \emph{saturation} of $Q$ and $M^{sm} \cong F/Q^{sat}$. 

Note that $\reg_k (M) = \reg_k (M^{sm})$ whenever $k \ge 1$ (see also \Cref{lem:modulo finite length} below). 
Moreover,  $H^0_\fm (M^{sm}) = 0$ gives 
\begin{align}
    \label{eq:reg of sm}
\reg (M^{sm}) = \reg_1 (M^{sm}) = \reg_1 (M). 
\end{align}

Using these observations, we make the connection between \Cref{thm:reg-1 modules} and \Cref{thm:Gotzmann reg} explicit. 

\begin{remark} 
   \label{rem:compare with Gotzmann} 
\Cref{thm:reg-1 modules} immediately implies Gotzmann's Regularity Theorem, \Cref{thm:Gotzmann reg}.  
Indeed, if $M = S/I$ has positive Krull dimension then $I^{sat} \neq S$ and \Cref{thm:reg-1 modules} yields 
\[
\reg I^{sat} = 1 + \reg (S/I^{sat}) = 1+  \reg_1 (S/I) \le b_1 (S/I), 
\]
which is exactly the estimate given in \Cref{thm:Gotzmann reg} because the number $s$ in that statement is equal to $b_1 (S/I)$ by \Cref{lem:compare hilb} as $e^+ (S/I) = 0$. 

Note that $b_1 (S/I^{sat}) = b_1 (S/I)$ by \Cref{lem:Macaulay cont for direct sum}(a) as $e^+ (S/I^{sat}) = e^+ (S/I) = 0$. 
\end{remark}

In order to bound the Castelnuovo-Mumford regularity of an arbitrary finitely generated graded module $M$ we need some further preparations. 

\begin{lemma}
   \label{lem:modulo finite length} 
If $L$  is any finite length submodule of $M$ then $\reg_k (M) = \reg_k (M/L)$ for any integer $k \ge 1$. 
\end{lemma}

\begin{proof}
This follows from the long exact cohomology sequence induced by the short exact sequence 
\[
0 \to L \to M \to M/L \to 0
\]
and using that $H^i_\fm (L) = 0$ if $i \ge 1$ as the Krull dimension of $L$ is zero by assumption. 
\end{proof}

The following observation will be useful. 

\begin{prop}
   \label{prop:compare reg}
For any finitely generated graded $S$-module $M$, one has 
\[
\reg (M) = \max \{r (M) - 1, \reg (M^{sm})\}. 
\]
\end{prop}

\begin{proof} 
Using Equation~\eqref{eq:reg of sm}, one gets
\begin{align}
    \label{eq:rewrite reg}
\reg (M) = \max \{ e(H^0_\fm (M), \reg_1 (M) \} = \max \{ e(H^0_\fm (M), \reg (M^{sm}) \}. 
\end{align}

Local cohomology modules can be used to measure the difference between the Hilbert function and the Hilbert polynomial (see \cite[Theorem 4.3.5]{BH-book}). For any integer $j$, one has 
\[
h_M (j) - p_M (j) = \sum_{i = 0}^d (-1)^i \dim_K [H^i_\fm (M)]_j,  
\]
where $d = \dim M$. So, the cohomological characterization of the regularity gives
\begin{align}
    \label{eq:RR special}
h_M (j) - p_M (j) =  \dim_K [H^0_\fm (M)]_j \quad \text{ if } j \ge \reg_1 (M^{sm}). 
\end{align}
Hence, if $r(M) - 1 \ge \reg (M^{sm})$ the definition of the regularity index $r(M)$ implies $e(H^0_\fm (M)) = r(M) -1$. Together with Equation~\eqref{eq:rewrite reg} this shows 
$ \reg (M) = r(M) - 1 =  \max \{ r(M) - 1, \reg (M^{sm}) \}$, as desired. 

If $r(M)  \le \reg (M^{sm})$ then Equation~\eqref{eq:RR special} implies $[H^0_\fm (M)]_j = 0$ whenever $ j \ge \reg (M^{sm})$, which means $e(H^0_\fm (M)) < \reg (M^{sm})$. Now Equation~\eqref{eq:rewrite reg} yields  
$ \reg (M) = \reg (M^{sm}) =  \max \{ r(M) - 1, \reg (M^{sm}) \}$, which completes the argument.  
\end{proof}

We are ready to show the announced regularity bound. 

\begin{thm}
    \label{thm:reg-0 modules}
If $M$ is any finitely generated graded $S$-module then $\reg (M) < b_0 (M)$. 
\end{thm} 

\begin{proof}
By \Cref{prop:reg index bound by Mac constants}, we know that $r(M) \le b_0 (M)$. \Cref{thm:reg-1 modules}  shows 
$\reg_1 (M) < b_1 (M)$, which is equivalent to  $\reg (M^{sm}) < b_1 (M)$. Since, by \Cref{prop:comp Macaulay const},  $b_0 (M) \ge b_1 (M)$   \Cref{prop:compare reg} proves the claim. 
\end{proof}

If $I \neq S$ is a homogeneous ideal of $S$ then the above result gives $\reg I = 1 + \reg (S/I) \le b_0 (S/I)$. In the case where the base field $K$ has characteristic zero this estimate also follows  by combining the main result of \cite{BS} (which uses different methods)  and \Cref{thm:Groebner degree bound}.

Recalling that $\reg (M) = \reg_0 (M)$, a comparison of Theorems \ref{thm:reg-1 modules} and \ref{thm:reg-0 modules} suggests an extension of these results. For establishing it, the following facts will be used.  

\begin{prop} 
   \label{prop:b_i and postive depth} 
If the base field $K$ is infinite then,    
for any  finitely generated graded $S$-module   $M$, one has: 
\begin{itemize}

\item[{\rm (a)}]  If $\ell \in [S]_1$ is any general linear form then  $b_i (M/\ell M) = b_{i+1} (M)$ for $i = 1,\ldots,\dim M -1$. 

\item[{\rm (b)}] If $\depth M \ge 1$ then $b_0 (M) = b_1 (M)$. 

\end{itemize}

\end{prop}

\begin{proof} 
Multiplication by $\ell$ on $M$ gives a short exact sequence
\[
0 \to (M/(0 :_M \ell))(-1) \stackrel{\times \ell}{\longrightarrow} M \to M/\ell M \to 0. 
\]
Since 
$K$ is  infinite by assumption  $0 :_M \ell$ has finite length by the generality of $\ell$. It follows that $M$ and $M/(0 :_M \ell)$ have the same Hilbert polynomial. Hence, using the above sequence and  \Cref{prop:Hilb pol with Macaulay const}, we get for the Hilbert polynomials 
\begin{align*}
p_{M/\ell M} (z) & = p_M (z) - p_M (z-1) \\
& = \sum_{j=2}^{\dim M} \bigg[\binom{z - e^+ (M)+j-2}{j-1} -  \binom{z-b_j (M)+j-2}{j-1} \bigg] \\
& = \sum_{i=1}^{\dim M -1 } \bigg[\binom{z - e^+ (M)+i-1}{i} -  \binom{z-b_{i+1} (M)+i-1}{i} \bigg].  
\end{align*}
Nakayma's lemma implies that the maximum degree of a minimal generator of $M/\ell M$ considered as an $S/\ell S$-module is equal to $e^+ (M)$. 
Using \Cref{prop:Hilb pol with Macaulay const} again, Claim (a) follows. 

For (b), combining \Cref{prop:compare reg and reg index} and \Cref{thm:reg-1 modules}, we obtain 
\[
r(M) \le 1 + \reg (M) = 1+  \reg_1 (M) \le b_1 (M), 
\]
and so $\max \{r(M), b_1 (M) \} = b_1 (M)$. Now  \Cref{prop:reg index bound by Mac constants} yields $b_0 (M) = b_1 (M)$. 
\end{proof}

One may wonder whether these statements can be extended to include $i = 0$ for (a) and Macaulay constants with larger indices for (b). For example, if $\depth M \ge 2$ one may hope that $b_0 (M) = b_1 (M) =  b_2 (M)$. However, in general this is not true.

\begin{example} 
    \label{exa:constants under hyperplane sections} 
Consider  the coordinate ring  $A = S/I_C$  of the cone $C \subset \PP^3$ over three non-collinear points in $\PP^2$, where $S = K[x_0,\ldots,x_3]$. It is Cohen-Macaulay of dimension two.  Its Hilbert function is 
\[
h_A (j) = \begin{cases}
3 j + 1 & \text{ if } j \ge 0; \\
0  & \text{ if } j < 0. 
\end{cases}
\]
Hence any 0-exact cone decomposition of $A$ is an isomorphism of $K$-vector spaces of the form
\[
A \cong x_0 K \oplus x_0^2 K \oplus x_0^3 K[x_0]  \oplus x_0^2 K[x_0, x_1] \oplus x_0 K[x_0,x_1] \oplus K[x_0, x_1], 
\]
which gives  
\[
b_0 (A) = b_1 (A) = 4 > 3 = b_2 (A). 
\]
Since $\depth A$ = 2, this shows that \Cref{prop:b_i and postive depth}(b) cannot be extended to an equality of more Macaulay constants if $M$ has large depth. 

Moreover,  if $\ell \in S$ is a general linear form then the Hilbert function of $A/\ell A$ is 
\[
h_{A/\ell A} (j) = \begin{cases}
3  & \text{ if } j \ge 2; \\
1  & \text{ if } j = 0; \\
0  & \text{ if } j < 0. 
\end{cases}
\]
Hence any 0-exact cone decomposition of $A/\ell A$ is of the form 
\[
A/\ell A \cong x_0 K  \oplus x_0^2 K[x_0]  \oplus x_0 K[x_0] \oplus K[x_0]. 
\]
Thus, 
$b_0 (A/\ell A) = b_1 (A/\ell A) = 3 < 4 = b_1 (A)$, and so the analog of \Cref{prop:b_i and postive depth}(a) is false for $i = 0$. 
\end{example}

We are ready to show  the main result of this section. It extends and subsumes Theorems \ref{thm:reg-1 modules} and \ref{thm:reg-0 modules}. 

\begin{thm}
   \label{thm:reg-i bound} 
If $M \ne 0$ is any finitely generated graded $S$-module then one has 
 \[
 \reg_i (M) < b_i (M) \quad \text{  for } i=0,1,\ldots,\dim M.  
 \]
 
 Moreover, for any integers $0 \le t \le d \le n$ and $b_t  >  b_{t+1}\ge \cdots \ge b_d > b_{d+1}$, there a is cyclic graded $S$-module $M$ satisfying $\dim M = d$, $\depth M = t$, $b_i (M) = b_i$ for $i = t,t + 1,\ldots,d+1$ and 
 \[
 \reg_i (M) = b_i (M) - 1 \quad \text{whenever } t \le i \le d. 
 \]
 For example, $M$ can be chosen as $M = (S/I) (- b_{d+1})$, where $I = I (f_t,\ldots,f_d)$ is the following ideal of $S$, 
 \begin{align}
      \label{eq:extremal ideal}
 I (f_t,\ldots,f_d) & = (\ell_{d-t},\ldots, \ell_{n-t-1}, \, f_d \ell_{0}, f_d f_{d-1} \ell_{1},\ldots, f_d \cdots f_{t+1} \ell_{d-t-1}, f_d  
 \cdots, f_{t}),  
 \end{align}
with linear forms $\ell_0,\ldots,\ell_{n-t-1}$ and non-zero homogeneous polynomials $f_t,\ldots,f_d \in S$ and $\deg f_i = b_i - b_{i+1}$  such that (as indicated)  $I (f_t,\ldots,f_d)$ has $n-t+1$ minimal generators. (Note that in the case $d = t$ the ideal $I(f_d)$ is simply defined as $I(f_d) = (\ell_{0},\ldots, \ell_{n-d-1}, f_d)$.)  

\end{thm}

\begin{proof}
First, we establish the estimates. 
Since $k$-regularities and Macaulay constants do not change when on extends the base field $K$ we may assume that $K$ is infinite. 
If $i = 0$ 
the claimed inequality is  true by  \Cref{thm:reg-0 modules}. 
Thus,  it suffices to consider $i \ge 1$,  and we may assume $d = \dim M \ge 1$. We use induction on $d$. 
If $d = 1$ we are done by \Cref{thm:reg-1 modules}. Let $\ell \in S$ be a general linear form.  The long exact cohomology sequence to the  exact sequence
\[
0 \to (M/(0 :_M \ell))(-1) \stackrel{\times \ell}{\longrightarrow} M \to M/\ell M \to 0 
\]
implies, for any integer $k \ge 1$, that $\reg_k (M) \le \reg_{k-1} (M/\ell M)$ (see, e.g., \cite[Lemma 2]{N-90}). 
Hence, if $i \ge 2$ we get by applying the induction hypothesis to $M/\ell M$, 
\begin{align*}
\reg_i (M) & \le \reg_{i-1} (M/\ell M) < b_{i-1} (M/\ell M) = b_i (M), 
\end{align*}
where we used \Cref{prop:b_i and postive depth} for the final equality. For $i = 1$, the desired estimate is true by  \Cref{thm:reg-1 modules}.

Second, we prove the claimed simultaneous sharpness of the bounds. Using \Cref{lem:constants and shifting} and passing to $M(-b_{d+1})$, it is enough to consider the case where $b_{d+1} = 0$. Then we have $b_{d+1} (S/I) = e^+ (S/I) = 0 = b_{d+1}$, as desired. 
We now determine the other Macaulay constants of $S/I$. Using induction on $d \ge 0$, we show that $S/I$ has a  $0$-exact cone decomposition of the form
\begin{align}
   \label{eq:cone decomp of extremal} 
S/ I (f_t,\ldots,f_d) \cong \bigoplus_{j= t}^d \  \bigoplus_{k = b_j+1}^{b_j - 1} x_{d-j}^k K[x_{d-j+1},\ldots,x_d], 
\end{align} 
and so $b_i (S/I) = b_i$ whenever $t \le i \le d$. 
If $t = d$ then  $I = I(f_d) = (\ell_{0},\ldots, \ell_{n-d-1}, f_d)$. Hence,  \Cref{exa:Mac constants principal} implies the desired cone decomposition as $b_0 = b_0 - b_1  = \deg f_0$.

Assume $t < d$ and set 
\[
\widetilde{d} = \max \{ i \le d-1 \; \mid \; a_i \ge 1\}. 
\]
Since $a_d$ and $a_t = b_t - b_{t+1}$ are positive by assumption, one has $t \le \widetilde{d} \le d-1$. We may assume that $f_{d-1} = \cdots = f_{\widetilde{d} + 1} = 1$. it follows that 
\begin{align*}
 I (f_t,\ldots,f_d) : f_d \\
& \hspace{-2.7cm} = (\ell_{d-t},\ldots, \ell_{n-t-1}, \, \ell_0,\ldots,\ell_{d- \tilde{d} -1},  \, 
 f_{\tilde{d}} \ell_{d - \tilde{d}}, f_{\tilde{d}} f_{\tilde{d}-1} \ell_{d - \tilde{d} +1},\ldots, f_{\tilde{d}} \cdots f_{t+1} \ell_{\tilde{d}-t-1}, f_d   \cdots, f_{t}). 
\end{align*}
The right-hand side is an ideal of the form $I (f_t,\ldots,f_{\tilde{d}}$. Thus, we a short exact sequence 
\[
0 \to (S/I (f_t,\ldots,f_{\tilde{d}} )(-a_d) \to S/I (f_t,\ldots,f_d) \to S/I(f_d) \to 0. 
\]
By the choice of $\widetilde{d}$, we have $\deg f_{\tilde{d}} \ge 1$. Hence, the induction hypothesis applies to the $\widetilde{d}$-dimensional algebra $S/I (f_t,\ldots,f_{\tilde{d}})$ and gives a 0-exact cone decomposition of the form 
\[
S/I (f_t,\ldots,f_{\tilde{d}})  \cong \bigoplus_{j= t}^{\tilde{d}} \  \bigoplus_{k = b_j+1}^{b_j - 1} x_{\tilde{d}-j}^k K[x_{\tilde{d}-j+1},\ldots,x_{\tilde{d}}]
\]
We know a 0-exact cone decomposition of  $S/I(f_d)$ by \Cref{exa:Mac constants principal}. Thus, the above exact sequence implies the claimed cone decomposition of $S/I (f_t,\ldots,f_d)$ by taking also into account that 
\[
[(x_i^e K[U]) (-s)]_j = x_i^e [K[U]]_{j - e - s} = [x_i^{e+s} k[U]]_j 
\]
for any integers $e$ and $s$ and any subset $U$ of the set of variables $\{x_0,\ldots,x_n\}$. Thus, we have shown $b_i (S/I) = b_i$ if $t \le i \le d$. 

By \cite[Lemma 3.4]{N-04}, we get for the local cohomology modules of $A = S/I (f_t,\ldots,f_d)$ that 
\[
H^i_{\fm} (A)^\vee \cong \begin{cases}
(S/(\ell_{0},\ldots, \ell_{n-i-1}, f_i))(b_i - i - 1) & \text{ if } t \le i \le d; \\
0 & \text{ otherwise}. 
\end{cases}
\]
It follows that $e(H^i_{\fm} (A)) = b_i -i -1$ if $H^i_{\fm} (A) \neq 0$, which implies $\reg_i (A) = b_i - 1 = b_i (A) - 1$ if $t \le i \le d$, completing the argument. 
\end{proof}

\begin{remark}
   \label{rem:Mac constants of extremal ideals} 
(i) Using the notation of \Cref{thm:reg-i bound}, the cone decomposition \ref{eq:cone decomp of extremal} gives that 
\[
b_i (S/I (f_t,\ldots,f_d)) = b_t \quad \text{whenever }  0 \le i \le t. 
\]

(ii) Combined with \Cref{lemdef:Macaulay constants},  \Cref{thm:reg-i bound} shows that, for any finitely generated graded $S$-module $M$, an upper bound for its Castelnuovo-Mumford regularity and, more generally, for any $i$-regularity of $M$ can be read off from the knowledge of the Hilbert function of $M$. 
\end{remark}

Using the first observation, it follows that the estimates on the $i$-regularity in \Cref{thm:reg-i bound} are also sharp if $b_t = b_{t+1}$. 

\begin{cor}
    \label{cor:sharpness} 
For any integers $0 \le t \le d \le n$ and $b_t  \ge  b_{t+1}\ge \cdots \ge b_d > b_{d+1}$, there is a cyclic graded $S$-module $M$ satisfying $\dim M = d$,  $b_i (M) = b_i$ for $i = t,t + 1,\ldots,d+1$ and 
 \[
 \reg_i (M) = b_i (M) - 1 \quad \text{whenever } t \le i \le d. 
 \]
\end{cor} 

\begin{proof}
Set $\tilde{t} = \max \{i \; \mid \; b_i = b_t \}$. Note that $\tilde{t} \le d$ as $b_d > b_{d+1}$ by assumption. The definition of $\tilde{t}$ implies $b_{\tilde{t}} > b_{\tilde{t} + 1}$. Hence, \Cref{thm:reg-i bound} shows that there is a cyclic $S$-module $M$ satisfying $\dim M = d$, $\depth M = \tilde{t}$, $b_i (M) = b_i$ for $i = \tilde{t},\ldots,d+1$ and 
 \[
 \reg_i (M) = b_i (M) - 1 \quad \text{whenever } \tilde{t} \le i \le d. 
 \]
Since $\depth M = \tilde{t}$ we get $\reg_i (M) = \reg_{\tilde{t}} (M) = b_{\tilde{t}}$ if $t \le i \le \tilde{t}$. By \Cref{rem:Mac constants of extremal ideals}, one has $b_i (M) = b_i = b_{\tilde{t}}$ if $t \le i \le \tilde{t}$, where the second equality is true by the definition of $\tilde{t}$. This completes the argument.  
\end{proof}

We close out this section with a consequence for sheaves. Let $\Fc \neq 0$ denote a coherent sheaf on $\PP^n = \Proj (S)$. We write 
$H^i (\Fc (j))$ for the $i$-th sheaf cohomology $H^i (\PP^n, \Fc (j))$ and  $h^i (\Fc (j))$  for $\dim_K H^i (\Fc (j))$. Following Mumford (see \cite[Lecture 14]{Mum}), the Castelnuovo-Mumford regularity of $\Fc$ is defined as 
\[
\reg \Fc = \min \{ j \in \Z \; \mid \; H^i (\Fc (j-i)) = 0 \text{ whenever } i \ge 1\}. 
\]

\begin{cor}
   \label{cor:reg sheaves}
If $\Fc$ is a coherent sheaf on $\PP^n$ of dimension $d-1 \ge 1$ that is generated by its global sections then there are uniquely determined integers $b_1 \ge \cdots \ge b_d >  0$ such that one has, for every integer $j \gg 0$, 
\begin{align}
    \label{eq:global sections}
h^0 (\Fc (j)) =  \sum_{j=1}^d \bigg[\binom{z +j-1}{j} -  \binom{z-b_j+j-1}{j} \bigg].
\end{align}
The Castelnuovo-Mumford regularity regularity of $\Fc$ satisfies 
\[
\reg (\Fc ) < b_2.   
\]

 Moreover, for any integers $2 \le  d \le n$ and $b_2  \ge \cdots \ge b_d > 0$, there a is a projective subscheme $Y$ of $\PP^n$ such that $\Fc = \mathcal{O}_Y$ satisfies Equation \eqref{eq:global sections} and 
 \[
 \reg (\Fc ) = b_2 - 1.  
 \]
\end{cor}

\begin{proof}
Define a graded $S$-module $M$ as $M = \oplus_{j \ge 0} H^0 (\Fc (j))$. It is a finitely generated $S$-module of dimension $d$ whose sheafification is $\widetilde{M} \cong \Fc$. The assumption on global sections implies that $e^+ (M) = 0$. Hence, by \Cref{thm:char hilb pol}, the Hilbert polynomial of $M$ can be uniquely written as 
\[
p_M (j) = \sum_{j=1}^d \bigg[\binom{z +j-1}{j} -  \binom{z-b_j+j-1}{j} \bigg], 
\]
where the integers $b_1 \ge \cdots \ge b_d >  0$ are the Macaulay constants $b_i = b_i (M)$ of $M$. 
By definition of $M$, one has $h^0 (\Fc (j)) = h^0 (\widetilde{M} (j))$ if $j \ge 0$. Hence $h^0 (\Fc (j)) = h_M (j) = p_M (j)$ if $j \gg 0$, which proves the first statement. 

Consider the exact sequence of graded $S$-modules (see, e.g., \cite[Corollary A1.12]{Eis05}), 
\[
0 \to H^0_{\fm} (M) \to M \to \oplus_{j \in Z} H^0 (\Fc (j)) \to H^1_{\fm} (M) \to 0. 
\]
By definition of $M$, the map in the middle is the identity in each degree $j \ge$. Since the initial degree of $M$ is zero, it follows that $H^0_{\fm} (M) = 0$ and $[H^1_{\fm} (M)]_j = 0$ if $j \ge 0$. 

Hence $\reg_2 (M) < b_2$ with $b_2 > e^+ (M) = 0$ (see \Cref{prop:comp Macaulay const}) implies $\reg (M) < b_2$. Now, \cite[Proposition 4.16]{Eis05}  gives $\reg (\Fc) < b_2$. 

Finally, denote by $Y$ the subscheme of $\PP^n$ defined by an ideal  $I (f_2,\ldots,f_d)$  as in \Cref{thm:reg-i bound}  with non-zero homogeneous polynomials $f_2,\ldots,f_d \in S$, where  $\deg f_i = b_i - b_{i+1}$. By \Cref{thm:reg-i bound}  and the proof of \Cref{cor:sharpness}, the sheaf $\Fc = \mathcal{O}_Y$ satisfies Equation \eqref{eq:global sections} and $ \reg (\Fc ) = b_2 - 1$.  
\end{proof}

Using \Cref{lem:compare hilb} and the notation of the above result, \cite[Proposition 4.1]{Dellaca} may be interpreted as showing that $\reg \Fc \le b_1$.  \Cref{cor:reg sheaves} is a strengthening of this estimate as $b_2 \le b_1$.  Notice that, by \Cref{thm:reg-i bound}, $b_2 - b_1$ can be arbitrarily large. 


\end{document}